\documentclass[reqno, 12pt]{amsart}
\usepackage{amsmath,amssymb,amscd}
\usepackage[centering, textwidth=6in,textheight=9in]{geometry}
\usepackage[percent]{overpic}
\usepackage{parskip}

\usepackage[colorlinks, hyperindex, linkcolor=blue, citecolor=blue, urlcolor=blue]{hyperref}

\newtheorem{thm}{Theorem}[section]
\newtheorem{cor}[thm]{Corollary}

\newtheorem{pro}[thm]{Proposition}

\theoremstyle{definition}
\newtheorem{definition}[thm]{Definition}
\newtheorem{remark}[thm]{Remark}
\newtheorem{con}[thm]{Conjecture}

\numberwithin{equation}{section}

\newcommand{\bC}{\mathbb{C}}

\newcommand{\bT}{\mathbb{T}}

\newcommand{\bZ}{\mathbb{Z}}

\newcommand\cF{\mathcal{F}}
\newcommand\cG{\mathcal{G}}

\newcommand\cN{\mathcal{N}}
\newcommand\cO{\mathcal{O}}
\newcommand\cP{\mathcal{P}}

\newcommand\cR{\mathcal{R}}

\newcommand\cU{\mathcal{U}}
\newcommand\cV{\mathcal{V}}
\newcommand\cW{\mathcal{W}}
\newcommand\cX{\mathcal{X}}

\newcommand\pa{\partial}

\newcommand\eps{\epsilon}

\newcommand\Diff{\mathrm{Diff}}
\newcommand\PH{\mathrm{PH}}

\begin{document}

\title[Homoclinic points of symplectic partially hyperbolic systems]{Homoclinic points of symplectic partially hyperbolic systems with 2D center}

\author[P. Zhang]{Pengfei Zhang}
\address{Department of Mathematics, University of Oklahoma, Norman, OK 73069 USA}
\email{pengfei.zhang@ou.edu}

\subjclass[2000]{37C29, 37C20, 37D30}

\keywords{generic properties, hyperbolic periodic point, homoclinic points,
partial hyperbolicity, center foliation}

\begin{abstract}
We consider a generic symplectic partially hyperbolic diffeomorphism close to
direct/skew products of symplectic Anosov diffeomorphisms with area-preserving diffeomorphisms
and prove that  every hyperbolic periodic point has transverse homoclinic points.
\end{abstract}

\maketitle

\section{Introduction}

Let $f \colon  M\rightarrow M$ be a diffeomorphism on a closed manifold $M$, 
$p$ be a hyperbolic periodic point of $f$,
and $W^{s,u}(p)$  be the stable and  unstable manifolds of $p$, respectively.
A point $x\in W^s(p)\cap W^u(p)\backslash \{p\}$ is called a  homoclinic point of $p$,
and the intersection $W^s(p)\cap W^u(p)$ at a homoclinic point $x$ is said to be   transverse
if $T_x W^s(p) + T_x W^u(p)=T_x M$. 
The fundamental importance of transverse homoclinic points  was first uncovered by Poincar\'e 
in the study of the restricted three-body problem \cite{Poin90, Poin}.
It is further revealed by Birkhoff \cite{Bir35} that there exist infinitely many
hyperbolic periodic points accumulating on any transverse homoclinic point.
Smale introduced in \cite{Sma65} a geometric model, 
now called  Smale horseshoe, for the dynamics around 
a transverse homoclinic point, and started
a systematic study of  general hyperbolic sets.

Xia and Zhang proved an interesting result in \cite{XZ06}  that periodic points are dense 
for a $C^r$-generic symplectic partially hyperbolic diffeomorphism 
close to a direct product of a symplectic Anosov diffeomorphism 
with an area-preserve diffeomorphism.
In this paper we obtain the existence of homoclinic points of hyperbolic periodic points of such systems.
Let $(N, \omega)$ be a closed symplectic manifold, $S$ be a closed surface with an area-form $\mu$.
Then $\Omega=\omega\oplus\mu$ is a symplectic form on the product manifold $M=N\times S$.
Let $f \colon  N\to N$ be a symplectic Anosov diffeomorphism,
$g \colon S \to S$ be an area-preserving diffeomorphism such that the direct product $f\times g$ is partially hyperbolic with center bundle $E^c_{(x,s)}=\{0_{x}\}\times T_{s}S$.
Replacing $f$ by $f^n$ for a larger $n$ if necessary, 
we may assume $f\times g$ is $4$-normally hyperbolic on $M$ (see Section \ref{ss.pahy} for more details).
Then there exists a $C^1$ open neighborhood $\cU$ of $f\times g$  
such that each map $\Phi \in\cU$ is partially hyperbolic, $4$-normally hyperbolic, dynamically coherent 
and plaque expansive (see Section \ref{ss.dcpe}). 
Moreover, the center foliation $\cF^c_\Phi$
is leaf conjugate to the trivial foliation $\cF^c_{f\times g} =\{\{x\}\times S: x\in N\}$.
Therefore, the center leaf $\cF^c_{\Phi}(p)$ is diffeomorphic to the surface $S$ for each $p\in M$.
Our first result is
\begin{thm}\label{thm:prod}
Suppose $r\ge 1$, $f\colon  N\to N$ be a  $C^r$ symplectic Anosov diffeomorphism,
$g \colon  S \to S$ an area-preserving diffeomorphism such that $f\times g$ is partially hyperbolic and $4$-normally hyperbolic.
Then there is a $C^1$-open neighborhood $\cU\subset\Diff^r_{\Omega}(M)$ of $f\times g$ 
such that for a $C^r$-generic map $\Phi \in \cU$, 
every hyperbolic periodic point of $\Phi$ admits transverse homoclinic points. 
\end{thm} 

Recall that a property (P) holds $C^r$-generically in a given $C^1$-open subset 
$\cU\subset\Diff^r_{\Omega}(M)$   if there is a $C^r$-residual subset 
$\cR\subset \cU$ such that every $f \in \cR$ satisfies the property (P).

More generally, let $f \colon  N\to N$ be a symplectic Anosov diffeomorphism,
and $g \colon  N \to \Diff^r_{\mu}(S)$ be a $C^r$ smooth cocycle.
This induces a skew product on the manifold $M= N\times S$ by $(f, g): M\to M$, $(x,s)\mapsto(f(x), g(x)(s))$.
Replacing $f$ by $f^n$ for large enough $n$ if necessary, 
we may assume $(f, g)$ is partially hyperbolic and $4$-normally  hyperbolic.
Our main result is
\begin{thm}\label{main}
Suppose $r\ge 1$,  $f \colon  N\to N$ be a  $C^r$ symplectic Anosov diffeomorphism,
$g \colon  N \to \Diff^r_{\mu}(S)$ be a $C^r$ smooth cocycle
such that the skew-product $(f, g)$ is partially hyperbolic and $4$-normally hyperbolic.
Then there is a $C^1$-open neighborhood $\cU \subset\Diff^r_{\Omega}(M)$ of $(f, g)$ 
such that for a $C^r$-generic map $\Phi\in\cU$, 
every hyperbolic periodic point of $\Phi$ admits transverse homoclinic points. 
\end{thm}

Our results are related to the long-standing conjecture of Poincar\'e  \cite{Poin}
(Chapter 3 in Vol. 1 and Chapter 33 in Vol. 3).
\begin{con}
Let $(M, \omega)$ be a closed symplectic manifold
and $\mathrm{Diff}^r_\omega(M)$ be the set of $C^r$ symplectic diffeomorphisms on $M$.
Then the following hold for a generic $f\in\mathrm{Diff}^r_\omega(M)$:
\begin{enumerate}
\item[(P1)]  The set of periodic points of $f$ is  \emph{dense} in the space $M$.

\item[(P2)]  There are transverse homoclinic points for every hyperbolic periodic point of $f$.
\end{enumerate}
\end{con}

The above conjecture is closely related to the  \emph{Closing  Lemma} and  \emph{Connecting Lemma}.
See \cite{Pug, PuRo} for the proof of $C^1$ Closing Lemma,
and \cite{Hay} for the proof of  $C^1$ Connecting Lemma.
Among $C^1$ diffeomorphisms, (P1) was proved by Pugh \cite{Pug2},
(P2) was proved by Takens \cite{Tak}, and a stronger version of (P2) was proved by Xia \cite{Xia96}. 
There are a few results for diffeomorphisms of higher regularity, most of which are  on surfaces.
More precisely, 
(P1) has been proved in \cite{CGPZ, EdHu},
(P2) has been proved on $S^2$ by Pixton \cite{Pix},
on $\bT^2$ by Oliveira \cite{Oli1}, and on surfaces of higher genus by Le Calvez and Sambarino \cite{LCS}.
See \cite{AsIr, FrLC, GaSh, KLCN, Mar16, QuXi, Xia06}  for results related to (P1), 
\cite{Oli2, Rob73, Xia2} for results related to (P2)
and  \cite{CoOl22, Irie15, XZ14, XZ17, Zha} for related results on dynamical systems of geometric origin.

\noindent{\bf Organization of the paper.}
In Section \ref{sec:pre} we introduce definitions and preliminary results.
In Section \ref{sec.pert} we construct a perturbation to change the center-leaf twist coefficient 
for a nonhyperbolic periodic point.
The proof of Theorem \ref{main} is given in Section \ref{sec.main}. 
It is clear that Theorem \ref{thm:prod} is a special case of Theorem \ref{main}.

\vskip.2in

\section{Preliminaries}\label{sec:pre}

We introduce necessary definitions and preliminary results that will be needed later.

\subsection{Partial hyperbolicity}\label{ss.pahy}
Let $M$ be a closed manifold and $\Diff^r(M)$ be the set of $C^r$ diffeomorphisms on $M$.
There are many different versions of definitions for a map $f\in \Diff^r(M)$ being partially hyperbolic \cite{HPS}.
Our definition follows \cite{BW10}. See \cite{Gou} for results about dominated splittings.
\begin{definition}\label{de:ph}
A diffeomorphism $f : M\to M$ is said to be \emph{partially hyperbolic} 
if there exist  a $Df$-invariant splitting $TM=E^s\oplus E^c \oplus E^u$ 
and a Riemannian metric on $M$ for which we can choose 
four continuous positive functions $\nu$, $\hat \nu$, $\gamma$ and $\hat \gamma$ on $M$ 
with $\nu,\hat\nu<1$ and $\nu<\gamma<\hat\gamma^{-1}<\hat\nu^{-1}$, 
such that for any $x\in M$, for any unit vector $v\in T_xM$, 
\begin{alignat*}{3}
&\|D_x f(v)\|< \nu(x)  & \quad & \text{ if } v\in E^s_x, \\
\gamma(x)< &\|D_x f(v)\| < \hat\gamma(x)^{-1} &  \quad & \text{ if } v\in E^c_x,  \\
\hat\nu(x)^{-1}< &\|D_x f(v)\| &  \quad & \text{ if } v\in E^u_x.  
\end{alignat*}
The three subbundles $E^s$, $E^c$ and $E^u$ are called  the stable, the center and the unstable bundles, respectively.  
Moreover, $f$ is said to be \emph{Anosov} (or equivalently, uniformly hyperbolic) if $E^c=\{0\}$. 
\end{definition}

Let $\PH^r(M)$ be the set of $C^r$ partially hyperbolic diffeomorphisms on $M$.
Note that the stable bundle $E^s$ is {uniquely integrable}. Let $\cF^s$ be the
stable foliation of $f$, whose leaves $\cF^s(x)$ are $C^r$ immersed submanifolds.
The same holds for the unstable bundle $E^u$. Denote by $\cF^u$ the unstable foliation. 
However, the center bundle $E^c$ may be  non-integrable.

\begin{definition}\label{def.knh}
Let  $f\in\PH^r(M)$, and $\nu,\gamma, \hat\nu$ and $\hat\gamma$ be the functions 
given in Definition \ref{de:ph}. 
Then $f$ is said to be {\it $k$-normally hyperbolic}
if  $\nu<\gamma^{k}$ and $\hat\nu<\hat\gamma^{k}$.
\end{definition}
It follows from Definition \ref{de:ph} that every partially hyperbolic diffeomorphism is 
$k$-normally hyperbolic for some $k \ge 1$.

\subsection{Dynamical coherence and plaque expansiveness}\label{ss.dcpe}
A diffeomorphism $f \in \PH^r(M)$ is said to be \emph{dynamically coherent}
if the subbundles $E^c$, $E^{c}\oplus E^{s}$ and  $E^{c}\oplus E^{u}$ integrate to invariant foliations
$\cF^c$, $\cF^{cs}$ and $\cF^{cu}$ respectively,  $\cF^c$ and $\cF^{s}$ subfoliate  $\cF^{cs}$,
$\cF^c$ and $\cF^{u}$ subfoliate  $\cF^{cu}$. 
Note that there are different versions of definitions of dynamical coherence in the
literature. See \cite{BW08} for more details.

\begin{pro}[\cite{HPS, Pes04}]\label{pro.leaf.Ck}
Let $1\le k\le r$, $f\in\PH^r(M)$ be dynamically coherent and $k$-normally hyperbolic.
Then $\cF^c(p)$ is a $C^k$ smooth submanifold for each $p\in M$.
\end{pro}

Hirsh, Pugh and Shub \cite[\S 7]{HPS} introduced the property \emph{plaque expansiveness}
for dynamically coherent partially hyperbolic maps. Here we follow the definition given in \cite{PSW12}.
More precisely, given a $c$-dimensional foliation $\cF$ on a closed $d$-dimensional manifold $M$,
one can pick finitely many foliation boxes for $\cF$, say $\phi_i: D^c \times D^{d-c} \to M$, $1\le i\le I$, such that the corresponding half size foliation boxes $\phi_i(\frac{1}{2}D^c \times \frac{1}{2}D^{n-c})$, $1\le i\le I$,  cover $M$. Each piece $\phi_i(D^c\times \{y\})$ is called a plaque of $\cF$,
and together they form a plaquation  $\cP$ of the foliation $\cF$ which cover the leaves of $\cF$ in a uniform fashion.

Let  $f\in \PH^r(M)$ be dynamically coherent.
A sequence $(x_n)_{n\in\bZ}$ is said to be an $\eps$-pseudo orbit if $d(f(x_n), x_{n+1})<\eps$ for each $n\in\bZ$, and is said to respect the plaquation $\cP$ if for each $n\in\bZ$ there exists a plaque $P_n \in \cP$ containing both $f(x_n)$ and $x_{n+1}$.
Then $f$ is said to be {\it plaque expansive} 
if there exist $\eps >0$ and a plaquation $\cP$ of the center foliation $\cF^c_f$ such that 
for any two $\eps$-pseudo orbits $(p_n)_{n\in\bZ}$ and $(q_n)_{n\in\bZ}$ that respect $\cP$,
if $d(p_n, q_n) < \eps$ for all $n\in\bZ$,
then  for each $n\in\bZ$ there exists a plaque $P_n \in \cP$ containing both $p_n$ and $q_n$.
Plaque expansiveness can be viewed as a generalization of the expansiveness 
from hyperbolic systems to partially hyperbolic ones.

\begin{pro}\label{pro.plaque} \cite[Theorem 7.1]{HPS}
Suppose $f\in\PH^1(M)$ is dynamically coherent and  plaque expansive.
Then there is a $C^1$-neighborhood $\cU \subset \PH^1(M)$ of $f$ such that every map $g \in \cU$ is dynamically coherent and plaque expansive,
and $(g, \cF^c_g)$ is canonically leaf conjugate to $(f, \cF^c_f)$.
\end{pro}
See also Theorem 1 and Theorem A in \cite{PSW12}.
Recall that $(g, \cF^c_g)$ is canonically leaf conjugate to $(f, \cF^c_f)$ if there is a homeomorphism $h: M\to M$ such that $h(\cF^c_g(x)) = \cF^c_f(h(x))$ and  $h(\cF^c_g(gx)) = \cF^c_f(h(fx))$ for every $x\in M$.

Plaque expansiveness is a desirable property but is not easy to be verified directly.
There are some simple and sufficient conditions for plaque expansiveness. See \cite{PSW12} for more details.
\begin{pro}\label{pro.C1}\cite[Theorem 7.2]{HPS}
Suppose $f\in\PH^1(M)$  admits an invairant center foliation $\cF^c_f$. 
If  $\cF^c_f$ is a $C^1$ foliation, then $f$ is plaque expansive.
\end{pro}

\subsection{Symplectic diffeomorphisms}\label{ss.sydi}
A 2-form $\omega$ on an even-dimensional manifold $M$ is said to be \emph{symplectic} 
if it is nondegenerate  and closed.  A symplectic manifold is a pair $(M, \omega)$ where $M$ is a smooth manifold and $\omega$ is a symplectic form on $M$.
A map $f\in \Diff^r(M)$ is said to be symplectic if  $f^\ast\omega=\omega$.
Let $\Diff^r_\omega(M) \subset \Diff^r(M)$ be the set of symplectic diffeomorphisms on $M$.
For convenience, let $d_{C^r}(f, g)$ be the $C^r$-distance  between two diffeomorphisms $f, g \in \Diff^r_{\omega}(M)$
and $B_{C^r}(f, \eps)$ be the $\eps$-ball of diffeomorphisms $g \in \Diff^r_{\omega}(M)$ with $d_{C^r}(f, g)<\eps$.
The following result is proved by Zehnder \cite[Theorem 1]{Zeh}.
\begin{pro}\label{pro.Zeh}
Let $(M, \omega)$ be a symplectic manifold. Then the set $\Diff^{\infty}_{\omega}(M)$
is dense in $\Diff^{r}_{\omega}(M)$ for each $r\ge 1$.
\end{pro}

Let $E\subset TM$ be a continuous subbundle with  $\dim E=i$, that is, $\dim(E_x)=i$ for any $x\in M$. 
The \emph{symplectic complement}  of $E$, denoted by $E^\omega$, is a subbundle  of $TM$ with fiber
$E^\omega_x=\{v\in T_xM: \omega(v,w)=0 \text{ for any }w\in E_x\}$.
Clearly $\dim E^\omega=\dim M - \dim E$. 
A subbundle $E\subset TM$ is said to be \emph{isotropic} if $E\subset E^\omega$,  
is said to be \emph{coisotropic} if $E\supset E^\omega$,  
is said to be \emph{symplectic} if $E\cap E^\omega=0$, 
and  is said to be \emph{Lagrangian} if $E= E^\omega$.
A submanifold $S\subset M$ is said to be a symplectic submanifold if $T_x S$ is a symplectic subspace of $T_x M$ for every $x\in S$. In this case, the restriction $\omega|_{S}$ serves as the symplectic form on $S$.

\subsection{Symplectic partially hyperbolic systems}\label{sec:symp}
Let $(M, \omega)$ be a symplectic manifold, $\PH^r_\omega(M):= \Diff^r_\omega(M) \cap \PH^r(M)$
be the set of symplectic partially hyperbolic diffeomorphisms on $M$.
Note that there might be different ways of formulating the partially hyperbolic splitting of  a map $f\in \PH^r_\omega(M)$, see \cite[Section 3]{BW08} for some interesting examples.
As we will see below, the center bundle $E^c_f$ can always be chosen to be a symplectic subbundle of $TM$.

\begin{pro}[\cite{SX06}]  \label{pro.symp.dom}
Suppose $f\in\Diff^r_\omega(M)$ admits  a dominated splitting $TM=E\oplus F$ with $\dim E \le \dim F$.
Then $f$ is partially hyperbolic with $E^s=E$, $E^c=E^\omega\cap F$
and $E^u= (E^c)^\omega\cap F$.
Moreover, $E^s$ and $E^u$ are isotropic, 
$E^s\oplus E^u$ and $E^c$ are symplectic and are symplectic-complement to each other.
\end{pro}

From now on, the center bundle $E^c$ of a map $f\in \PH^r_\omega(M)$ is \emph{always} assumed to be symplectic.

\begin{remark}
It is proved in \cite{SX06} that symplectic partially hyperbolic
maps are symmetric. That is,
one can take $\hat\nu=\nu$ and $\hat\gamma=\gamma$ in Definition \ref{de:ph}.
Then the normal hyperbolicity condition in Definition \ref{def.knh}
for general partially hyperbolic maps
admits a simpler form in the symplectic case. 
That is, a map $f\in\PH^r_\omega(M)$ is $k$-normally hyperbolic 
if the functions $\nu$ and $\gamma$ in Definition \ref{de:ph} satisfy $\nu<\gamma^{k}$.
\end{remark}

\begin{pro}[\cite{XZ06}] \label{pro.symp.center}
Suppose $f\in \mathrm{PH}^r_\omega(M)$ is dynamically coherent. 
Then each center leaf $\cF^c(x)$ is a symplectic submanifold of $M$
with respect to the restricted symplectic form $\omega|_{\cF^c_f(x)}$.
Moreover, the restriction $f \colon  \cF^c_f(x)\to \cF^c_f(fx)$
is a symplectic diffeomorphism for every $x\in M$.
\end{pro}

A center leaf $\cF^c(x)$ is said to be periodic if $f^k\cF^c(x) =\cF^c(x)$ for some $k\ge 1$.
In  \cite{NT01} Nitic\v{a} and T\"or\"ok  proved the following.
\begin{pro}\label{pro.denseCL}
Suppose $f \in\PH^r_{\omega}(M)$ is dynamically coherent and plaque expansive. 
Then the periodic center leaves of  $\cF^c$ are dense in $M$.
\end{pro}

\subsection{Birkhoff normal form and nonlinear stability}\label{BNF}
Let $S$ be a closed surface, $\mu$ be an area form on $S$, 
$f \colon  S \to S$ be a $C^4$ symplectic map, 
and $p$ be an elliptic fixed point of $f$, that is, 
That is, $|\lambda_p|=1$ and $\lambda_p \neq \pm 1$, where $\lambda_p$ is an eigenvalue of the linear map $D_p f \colon  T_p S \to T_p S$.
Recall that an elliptic fixed point $p$ is said to be \emph{nonresonant} if $\lambda_p^j\neq1$ for each $1\le j\le 4$.
Birkhoff \cite{Bir22} showed that
there exist a unique real number $\tau_1$ and a symplectic embedding $h: U\to S$ on a neighborhood $U$ of $0\in \bC$ with $h(0) = p\in S$ such that 
\begin{align}
h^{-1}\circ f\circ h(z) = \lambda_p\cdot z\cdot e^{i \cdot \tau_1|z|^2} + O(|z|^{4}).\label{normal}
\end{align}
See also \cite[Theorem 2.12]{Mos77}. 
The number $\tau_1=\tau_1(f, p)$ is called the first twist coefficient of $f$ around the fixed point $p$,
the map $h$ is called the first-order Birkhoff normalization,
and the map of the form \eqref{normal} is called the first-order Birkhoff Normal Form of $f$ at $p$.

\begin{definition} 
An elliptic fixed point $p$ of a surface map
$f \colon  S\to S$ is said to be {\it nonlinearly stable},
if there is a sequence of nesting neighborhoods $\{D_n: n\ge 1\}$ of $p$
such that for each $n\ge 1$, $f(D_n)=D_n$ and the restriction of $f$
on $\partial D_n\simeq S^1$ is a transitive circle map.
\end{definition}

Note that nonlinearly stable periodic points are isolated from the dynamics
in the sense that it cannot be reached from any invariant curve 
whose starting point lies outside some $D_n$.
The following is Moser's \emph{Twisting Mapping Theorem} \cite{Mos62}. 
See also \cite[Theorem 2.13]{Mos77}. 
\begin{thm}\label{Moser}
Let $r\ge 4$, $f \in \Diff_{\mu}^{r}(S)$ and $p$ be a nonresonant elliptic fixed point of $f$.
If the first twist coefficient of $f$ at $p$ is nonzero, then $p$ is nonlinearly stable.
\end{thm}

\subsection{Homoclinic intersections for surface diffeomorphisms}\label{ss.gene}
Let $S$ be a closed surface of genus $g_S$, $\mu$ be an area form on $S$,
and $\cG^r_{\mu}(S) \subset \Diff^r_{\mu}(S)$ be the set of $C^r$ symplectic diffeomorphisms 
$f : S \to S$ satisfying the following conditions:
\begin{enumerate}
\item[(G1)]  Every periodic point of $f$ is either elliptic or hyperbolic.
 Moreover, if $p$ is an elliptic
periodic point of period $n$, then the eigenvalues of $Df^n(p)$ are not roots of unity.

\item[(G2)] Stable and unstable branches of hyperbolic points that intersect must also intersect
transversally.

\item[(G3)] Every elliptic periodic point of $f$ is nonlinearly stable.
\end{enumerate}

The following alternative condition for (G3) is used in \cite{LCS}:
\begin{enumerate}
\item[(G3)'] Every elliptic periodic point of $f$ admits a nesting sequence of topological disks
whose boundaries consist of finitely many pieces of stable and unstable manifolds of some hyperbolic periodic points.
\end{enumerate}
Both (G3) and (G3)' are sufficient when applying Mather's  prime-end theory \cite{Mat81, Mat82}.
The main difference between (G3) and (G3)' is that (G3)' is a generic condition for $C^r$, $r\ge 1$,
while (G3) is a generic condition only for $r\ge 4$. Since the existence of transverse homoclinic points is a $C^1$-open property for one periodic point and a $G_{\delta}$-property for diffeomorphisms, we only need to prove the $C^r$-denseness of diffeomorphisms with transverse homoclinic points for every hyperbolic periodic points. Combining with Proposition \ref{pro.Zeh}, we have that the $C^r$-generic
existence of transverse homoclinic points implies the $C^k$-generic existence for any $r\ge k$.
Therefore, one can replace (G3)' in  \cite{LCS} by (G3) when defining the set $\cG^r_{\mu}(S)$.

\begin{pro}\label{homo01}
Let $S=S^2$ or $\bT^2$, $f\in \cG^r_{\mu}(S)$.
Then there  are transverse homoclinic points  for any hyperbolic periodic point of $f$.
\end{pro}
The case  $S=S^2$ is proved by Pixton \cite{Pix}, and the case $\bT^2$ is proved by Oliveira \cite{Oli1}. 
See also \cite[Theorem 1.5]{LCS}.

Now we consider a closed surface $S$ of genus $g_S \ge 2$ and a map $f \in \cG^r_{\mu}(S)$.
Let $P(f) \subset S$ be the set of periodic points of $f$, and $P_h(f)\subset P(f)$ be the set of hyperbolic ones.
Le Calvez and Sambarino \cite{LCS} obtained several important characterizations for such maps.
Here we only list a few that are needed in this paper. 
See  \cite[Proposition 1.4, Theorem 1.5 and 1.6]{LCS} for more details.
\begin{pro}\label{pro.LCS}
Let $S$ be a closed surface of genus $g_S \ge 2$, $f \in \cG^r_{\mu}(S)$. Then $|P_h(f)| \ge 2g_S-2$. 
Moreover, the following dichotomy holds:
\begin{enumerate}
\item $|P_h(f)|> 2g_S-2$: every hyperbolic periodic point of $f$ has transverse homoclinic points;

\item $|P_h(f)|=2g_S-2$:  every periodic point of $f$ is hyperbolic, and each stable (resp. unstable) branch of every hyperbolic periodic point is dense on $S$.
\end{enumerate} 
\end{pro}

\subsection{Symplectic partially hyperbolic systems with 2D center}\label{ss.sph}
Let $\mathrm{PH}^r_\omega(M,2)$ be the set of symplectic partially hyperbolic diffeomorphisms
with 2D center bundles. As we have mentioned right after Proposition \ref{pro.symp.dom},
the center bundles will be assumed to be symplectic. 
Given a map $f\in \mathrm{PH}^r_\omega(M,2)$ and
a periodic point $p$  of minimal period $n$,
the splitting $T_pM=E^s_p\oplus E^c_p\oplus E^u_p$ at $p$ is $D_pf^n$-invariant. 
It follows from Definition \ref{de:ph} that the eigenvalues of $D_pf^n$ along the subspace $E^s_p$ (resp. $E^u_p$)
have modulus smaller (resp. larger) than $1$.
Moreover, it follows from Proposition \ref{pro.symp.center} that the two eigenvalues of $D_pf^n$ along the 2D symplectic subspace $E^c_p$ are of the from $\lambda_c(p, f^n)$ and $\lambda_c(p, f^n)^{-1}$.
Therefore, we have the following  dichotomy:
\begin{enumerate}
\item either $|\lambda_c(p, f^n)|\neq 1$: then $p$ is a hyperbolic periodic point of $f$; 

\item  or  $|\lambda_c(p, f^n)|=1$: then  $p$ is nonhyperbolic with a 2D neutral subspace.
\end{enumerate}
The stable manifold $W^s(p)$ of a periodic point $p$ of period $n$ (not necessarily hyperbolic)
is defined to be the $f^n$-invariant submanifold tangent to the generalized eigenspace of eigenvalues $\lambda$
of $D_p f^n$ with $|\lambda|<1$. It coincides with the stable leaf $\cF^s(p)$ when $p$ is nonhyperbolic
and strictly contains  the stable leaf $\cF^s(p)$ when $p$ is a hyperbolic periodic point. 
Note that   $W^s(p)$ may be thin along the center direction.
Given a positive number $\delta >0$, we can define the stable disk $W^s(p, \delta)$ centered at $p$ of radius $\delta$
with respect to the induced submanifold metric on $W^s(p)$.
Similarly one can define the unstable manifold $W^u(p)$ and the unstable disk $W^u(p, \delta)$.

\subsection{Kupka--Smale property}
Robinson \cite{Rob70} extended the Kupka--Smale property to symplectic diffeomorphisms.
For convenience, we will restrict to $\mathrm{PH}^r_\omega(M,2)$, which is an open subset of $\Diff_{\omega}^r(M)$.
Let $f\in \mathrm{PH}^r_\omega(M,2)$ and $p$ be a nonhyperbolic periodic point of minimal period $n$, that is, $|\lambda_c(p, f^n)|=1$, see \S~\ref{ss.sph}.
Then $p$ is said to be \emph{center-elliptic} if $\lambda_c(p, f^n)  \neq \pm 1$,
and is \emph{center-nonresonant} if  $\lambda_c(p, f^n)^k \neq 1$ for each $1\le k \le 4$. 
This generalizes of the definition of (nonresonant) elliptic periodic points given in Section \ref{BNF}. 
For each $n\ge 1$, let $P_n(f)$ be the set of points fixed by $f^n$.
Clearly $P_n(f)$ is a closed set. 
Robinson proved in \cite{Rob70} the following
\begin{pro}\label{rob-ks}
There exists a  $C^1$-open and $C^r$-dense subset $\cU_n^r\subset \mathrm{PH}^r_\omega(M, 2)$  such that 
for each $f\in \cU_n^r$, 
\begin{enumerate}
\item $P_n(f)$ is finite and varies continuously;

\item  each periodic point in $P_n(f)$ is either hyperbolic or center-nonresonant;

\item $W^u_f(p,n)\pitchfork W^s_f(q,n)$ (possibly empty) for any  $p, q\in P_n(f)$.
\end{enumerate}
\end{pro}

Let $\cR_{KS}(2)=\bigcap_{n\ge 1} \cU^r_n$, which is a $C^r$-residual subset
of $\mathrm{PH}^r_\omega(M, 2)$. It follows that 
\begin{cor}\label{rob-cor}
Let $f\in \cR_{KS}(2)$. Then the following hold:
\begin{enumerate}
\item the set of periodic points $P(f)$ is countable;

\item each periodic point of $f$ is either hyperbolic or center-elliptic.
Moreover, if $p$ is a center-elliptic
periodic point of period $n$, then the center-eigenvalues of $Df^n(p)$ are not roots of unity;

\item $W^u_f(p)\pitchfork W^s_f(q)$ (possibly empty) for any periodic points $p$ and $q$ of $f$.
\end{enumerate}
\end{cor}

\begin{remark}
The last item of the above Kupka--Smale property says  that, 
{\it when} $W^s_f(p)$ and $W^u_f(q)$ have a 
nontrivial intersection, the intersection is actually transverse.
However, it does not address the question
whether $W^s_f(p)$ and $W^u_f(q)$ can have any nontrivial intersection.
Theorem \ref{main} states that there are homoclinic intersections 
for every hyperbolic periodic point generically.
\end{remark}

\section{Perturbations of the center-leaf twist coefficients}\label{sec.pert}

Let $(M, \omega)$ be a closed symplectic manifold, $1\le k\le r$,
and $\cN^r_{k}(2) \subset \mathrm{PH}^r_\omega(M,2)$ be the set of symplectic partially hyperbolic diffeomorphisms on $M$
that are $k$-normally hyperbolic, dynamically coherent and plaque expansive. 
It follows from Definition \ref{def.knh} and Proposition \ref{pro.plaque}
that $\cN^r_{k}(2)$ is a $C^1$-open subset of $\mathrm{PH}^r_\omega(M,2)$.
Moreover, it follows from Proposition \ref{pro.leaf.Ck}, \ref{pro.symp.dom} and \ref{pro.symp.center}
that for each $f\in \cN^r_{k}(2)$, the center leaf $\cF^c_f(p)$ is a $C^k$ symplectic submanifold with symplectic form $\omega|_{\cF^c_f(p)}$, and the restriction $f \colon  (\cF^c_f(p), \omega|_{\cF^c_f(p)}) \to  (\cF^c_f(fp), \omega|_{\cF^c_f(fp)})$  is a $C^k$ symplectic diffeomorphism  for each $p\in M$. 
In order to apply Proposition \ref{homo01} and \ref{pro.LCS} to these center-leaf mappings, we need to show that $C^r$-generically, the conditions (G1)--(G3) hold for all periodic center leaves. Note that (G1) and (G2) have been established in Proposition \ref{rob-ks}, see also Corollary \ref{rob-cor}. For (G3),  it follows from Theorem \ref{Moser} that we only need to check the center-leaf twist coefficient $\tau_1\neq 0$ for nonhyperbolic periodic orbits.

\begin{pro}\label{pro:bir}
Suppose $r\ge 4$. 
Then there exists a  $C^4$-open and $C^r$-dense subset $\cV_n\subset \cN^r_{4}(2)$ 
such that for each $f\in \cV_n$ and each periodic point $p\in P_n(f)$, either $p$ is hyperbolic, 
or the center-leaf twist coefficient $\tau_1(p,f^k,\cF^c_f(p))\neq 0$,
where $k$ is the minimal period of the point $p$.
\end{pro}
\begin{proof}
Let $\cU^r_n(2)=\cN^r_{4}(2)\cap \cU_n^r$,
where $\cU_n^r$ is the $C^1$-open and $C^r$-dense subset of $\mathrm{PH}^r_\omega(M,2)$ given in Proposition \ref{rob-ks}.
Let $f\in\cU^r_n(2)$, and $p\in P_n(f)$ be a nonhyperbolic periodic point, 
and $k$ be the minimal period of $p$. Then $k|n$.

Since $f\in \cU^r_n$, the periodic point $p$ is center-nonresonant.
Let $h_c: U_c \to \cF^c_f(p)$ be the symplectic embedding of an open set $U_c\subset \bC$ given in Section \ref{BNF}
such that 
\begin{align}
h_c^{-1}\circ f^k|_{\cF^c_f(p)} \circ h_c(z) =\lambda_p z e^{i \tau_1 |z|^2} + O(|z|^4),
\end{align}
where $\tau_1 =\tau_1(p,f^k,\cF^c_f(p))$ be the first twist coefficient of the center-leaf map $f^k|_{\cF^c_f(p)}$ at $p$.

\noindent{\bf Claim.}
Let $\cU_f \subset \cU^r_n$ be a  $C^4$-open neighborhood of $f$
such that $P_n(\cdot)$ is a finite subset of the same cardinalty and varies continuously on $\cU_f$.
If $\tau_1(p,f^k,\cF^c_f(p))\neq 0$, then there exists a $C^4$-open neighborhood
$\cU(f,p)\subset \cU_f$ of $f$ such that $\tau_1(p_g,g^k,\cF^c_g(p_g))\neq 0$ for all $g\in\cU(f,p)$.

\begin{proof}[Proof of Claim]
Note that the periodic point $p$ is nondegenerate. Let $p_g$ be the continuation of $p$
for a map $g$ that is close to $f$.
Moreover, the partially hyperbolic splitting of the map $g$ varies continuously,
and $g$ admits a $g$-invariant center foliation $\cF^c_g$. 
Therefore,  the map $g\mapsto(g^k,\cF^c_g(p_g))$ varies continuously, 
so is the first twist coefficient $g\mapsto \tau_1(p_g,g^k,\cF^c_g(p_g))$.
This completes the proof of Claim.
\end{proof}

In the following we consider the case that  $\tau_1(p,f^k,\cF^c_f(p))=0$. 
We will add a small positive twist to the Birkhoff normal form on a small neighborhood of the center leaf at $p$.
More precisely, let $\eps$ and $\delta$ be two small positive numbers (to be specified later),
$b:[0, \infty) \to [0, 1]$ be a smooth bump function with $b(t)=1$ for $t \le 1/3$ and $b(t)=0$ for $t\ge 2/3$,
and  $\hat g_c$ an integrable twist map on an open ball $B_c(0, \eps) \subset U_c$  given by
\begin{align}
\hat g_c(z) = z e^{i \delta b(|z|/\eps)|z|^2}. \label{eq.add.tw}
\end{align}
Note that $\hat g_c(0)=0$, $\hat g_c(z)=z$ when $|z|\ge 2\eps/3$, 
and the $C^r$-norm of $\hat g_c - Id$ can be made arbitrarily small by reducing the parameter $\delta$.
Then consider the map $g_c: U_c \to U_c$ defined by $g_c=h_c\circ \hat g_c \circ h_c^{-1}$.
Note that $g_c$ is symplectic since both $h_c$ and $\hat g_c$ are symplectic.
Then it is easy to see that the Birkhoff coefficient $\tau_1(p;f^{k}\circ g_c,\cF^c_f(p)) =\delta\, b(0) >0$. 
Note that $k$ is the period of $p$, not necessarily the period of the center leaf $\cF^c_f(p)$.
In particular, it is possible that $f^j\cF^c_f(p)=\cF^c_f(p)$ for some $j|k$.
In this case, the intersection $\cO(p,f)\cap \cF^c_f(p)$ is a finite set,
and the support of $g_c$ can be made small enough such that it does not interfere
with the intermediate returns of $p$ to $\cF^c_f(p)$.
Note that the map $g_c$ has yet to be defined on $M\backslash \cF^c_f(p)$.

Next we will extend $g_c$ to the whole manifold $M$.
By Darboux's theorem \cite{MS17},
one can extend the local coordinate system $(x_1, y_1)$ on $U_c\subset \cF^c_f(p)$
to a local neighborhood  $U\subset M$ containing $U_c$, 
say $(x_i, y_i)_{1\le i \le d}$,
such that $p=(0,0,\dots, 0)$ and $\omega=\sum_i dx_i\wedge dy_i$, where $1\le i \le d$.
Suppose $g_c(x_1, y_1)=(X_1(x_1, y_1), Y_1(x_1, y_1))$, $(x_1, y_1) \in U_c$. 
It follows from the definition \eqref{eq.add.tw} that 
the support of the map $g_c$ is contained in the ball $B_c(0, \eps) \subset U_c$.
Note that both $h_c$ and $\hat g_c$ are close to identity, so is $g_c$.
It follows from \cite[Lemma 9.2.1]{MS17} that there exists a $C^{r+1}$-small function $V_c(X_1, y_1)$ supported on $B_c(0, \eps) \subset U_c$ such that 
$g_c(x_1, y_1)=(X_1, Y_1)$ if and only if 
\begin{align}
X_1 -x_1 =\frac{\pa V_c}{\pa y_1}(X_1, y_1), \quad Y_1 -y_1 = -\frac{\pa V_c}{\pa X_1}(X_1, y_1). \label{eq.generating}
\end{align}
Then we extend the above function $V_c$  to 
a $C^{r+1}$-small function $V$ supported on a small ball $B(0, \eps') \subset U$ with $V|_{U_c}=V_c$ 
(reducing $\eps$ and $\delta$ if necessary).
Let $g$ be the symplectic diffeomorphism on $U$ generated by the function $V$ using the vector form of the equation \eqref{eq.generating}: $g(x,y) =(X, Y)$ if and only if
\begin{align}
X_i -x_i =\frac{\pa V_c}{\pa y_i}(X, y), \quad Y_i -y_i = -\frac{\pa V_c}{\pa X_i}(X, y),
\quad 1\le i\le d. 
\end{align}
Note that $g$ is supported on $B(0, \eps') \subset U$. 
So we can extend $g$ to the whole manifold $M$ by setting  $g=Id$ on $M\backslash U$. 
It follows that $g$ is $C^r$-close to identity,
and $g=g_c$ on a small neighborhood of $p$ in $\cF^c_f(p)$. Let $\hat f=f\circ g$. 
Then we have 
 $\hat f^i(p)=f^i\circ g(p)=f^i (p)$ for each $1\le i\le k$, $\hat f^k(\cF^c_f(p))=\cF^c_f(p)$
and $\tau_1(p,\hat f^k, \cF^c_f(p))=\tau_1(p, f^{k}\circ h_c,\cF^c_f(p)) > 0$.
Note that an invariant normally hyperbolic manifold is isolated and persists under perturbations.
The fact $\cF^c_f(p)$ is a  normally hyperbolic manifold
of $\hat f^k$ implies that  $\cF^c_{\hat f}(p)=\cF^c_f(p)$.
Therefore, we can rewrite the above conclusion as $\tau_1(p,\hat f^k, \cF^c_{\hat f}(p)) > 0$.

As we have shown in the Claim, 
there is a $C^4$-open neighborhood $\cU(p, \hat f) \subset \cU$ of $\hat f$
such that for any $h\in\cU(p, \hat f)$, 
the continuation $p_h$ satisfies $\tau_1(p_h, h^k, \cF^c_h(p_h))\neq 0$.
Let $\ell=|P_n(\hat f)|$, which is constant on $\cU$.
Then by induction,
we can find a $C^4$-open subset $\cU^{(\ell)}_f \subset \cU(p, \hat f)$ arbitrarily $C^r$-close $f$, 
such that for each $h\in \cU^{(\ell)}_f$ 
and each periodic point $p\in P_n(h)$, either it is hyperbolic 
or  the center-leaf Birkhoff coefficient $\tau_1(p, h^k, \cF^c_h)\neq 0$,
where $k$ is the minimal period of the point $p$.

Note that the map $f$ is chosen arbitrarily in $\cU_{n}^r(2)$, 
and $\cU^{(\ell)}_f$ is a $C^4$-open set with elements $C^r$-close to $f$.
Putting these sets   $\cU^{(\ell)}_f$ together, we get a $C^4$-open
and  $C^r$-dense subset in $\cU_{n}^r(2)$, say $\cV_n$, such that for each $f\in \cV_n$
and each periodic point $p\in P_n(f)$, either $p$ is hyperbolic, 
or  $p$ is center-nonresonant and the center-leaf Birkhoff coefficient $\tau_1(p,f^k, \cF^c_f)\neq 0$,
where $k$ is the minimal period of $p$.
Then it follows that $\cV_n$ is a $C^4$-open 
and  $C^r$-dense subset of $\cN^r_4(2)$.
\end{proof}

\begin{pro}\label{pro:R}
Let $\cV_n$ be the  $C^4$-open and $C^r$-dense subset of $\cN^r_4(2)$ given in Proposition \ref{pro:bir},
and $\cR=\bigcap_n \cV_n$. Then $\cR$ contains a $C^r$-residual subset of $\cN^r_4(2)$ such that
for each $f\in\cR$,
\begin{enumerate}
\item $P_n(f)$ is finite, and each periodic point is either hyperbolic or 
center-nonresonant;

\item  $W^s(p)\pitchfork W^u(q)$ for any two hyperbolic periodic points $p,q$;

\item the center Birkhoff coefficient $\tau_1(p,f^k,\cF^c(p))\neq 0$ 
for each center-nonresonant periodic point $p$.
\end{enumerate}
\end{pro}

\section{Proof of the main theorem}\label{sec.main}

Let $(N, \omega)$ be a closed symplectic manifold, 
$f \colon N \to N$ be a symplectic Anosov diffeomorphism,
$S$ be a closed surface with an area form $\mu$,
and $g \colon  N\to \Diff^r_{\mu}(S)$ be a smooth cocycle such that the skew-product
$(f, g) \in \PH^r_{\Omega}(M)$ is $4$-normally hyperbolic, where $M=N\times S$
and $\Omega=\omega\oplus \mu$. Moreover, the center leaf $\cF^c_{(f, g)}(x, s)=\{x\}\times S$ 
for every $(x,s)\in M$.
It follows from Proposition \ref{pro.C1} that $(f, g)$ is dynamically coherent and plaque expansive.
Let $\cU_{\ast} \subset \PH^r_{\Omega}(M)$ be the $C^1$-neighborhood of $(f,g)$ given by Proposition \ref{pro.plaque}
such that every $\Phi \in \cU_{\ast}$ is $4$-normally hyperbolic, dynamically coherent and plaque expansive,
and $(\Phi, \cF^c_{\Phi})$ is canonically leaf conjugate to $((f,g), \cF^c_{(f, g)})$. In particular, 
$\cU_{\ast} \subset \cN^r_4(2)$, and
the leaves of the center foliation $\cF^c_{\Phi}$ are diffeomorphic to $S$.
Assume $r\ge 4$ for the moment.  
Then it follows from  Proposition \ref{pro.leaf.Ck} and \ref{pro.symp.center} that the center leaf $\cF^c_{\Phi}(x)$ is a $C^4$ symplectic  submanifold diffeomorphic to $S$,
and the restriction  $\Phi: \cF^c_{\Phi}(x) \to \cF^c_{\Phi}(\Phi x)$ is a $C^4$ symplectic diffeomorphism  for every $x\in M$.

We will divide the remaining of the proof of Theorem \ref{main}
into two cases depending on the genus of the surface $S$.
The proof when $S=S^2$ or $\bT^2$ is  easier mainly due to Proposition \ref{homo01}.
We give a proof of Theorem \ref{main} in these two special cases first
and deal with the general cases later.

\begin{proof}[Proof of Theorem \ref{main}. Part 1.]
Let $\cU_n^r$ be the subset given in Proposition \ref{rob-ks}, 
$\cV_n \subset \cN^r_4(2) \cap \cU_n^r$ be the subset given in Proposition \ref{pro:bir} and
$\cR=\bigcap_n \cV_n$. Then $\cR_{\ast}:=\cR \cap \cU_{\ast}$ is a residual subset of $\cU_{\ast}$.
 
Let $\Phi \in \cR_{\ast}$. Then for any hyperbolic periodic point 
$p$ of $\Phi$ with minimal period $n$, the center leaf $\cF^c_{\Phi}(p)$ is diffeomorphic to the surface $S$ and is invariant under $\Phi^n$, where $n$ is the minimal period of $p$.  
It follows from Theorem \ref{Moser} and Proposition \ref{pro:R} that every elliptic periodic point of  the center-leaf map 
$\Phi^n: \cF^c_{\Phi}(p) \to \cF^c_{\Phi}(p)$ is nonlinearly stable.
Combining with Proposition \ref{rob-ks}, we have that
the map $\Phi^n|_{\cF^c_{\Phi}(p)}$ satisfies all three conditions (G1)--(G3) given in Section \ref{ss.gene}.
That is, $\Phi^n|_{\cF^c_{\Phi}(p)} \in \cG^4_{\Omega|_{\cF^c_{\Phi}(p)}}(\cF^c_{\Phi}(p))$. 
Then it follows from Proposition \ref{homo01} that the hyperbolic periodic point $p$ admit transverse homoclinic points with respect to the surface map  $\Phi^n|_{\cF^c_{\Phi}(p)}$.
These points are also   transverse homoclinic points of $p$ for $\Phi$ on the ambient manifold $M$.
This holds for any hyperbolic periodic point $p$ and for any map $\Phi\in\cR_{\ast}$.
So Theorem \ref{main} holds for every $r\ge 4$ when $S=S^2$ or $\bT^2$.
The $C^r$-generic existence of transverse homoclinic points with $1\le r\le 3$ follows directly from the $C^4$-generic existence since it is  a $G_{\delta}$ property.
\end{proof}

In the case $S=S^2$ or $\bT^2$, no secondary perturbation is needed during the proof of Theorem \ref{main}.
Next we deal with the remaining case that $S$ is a closed surface of genus $g_S\ge 2$.

\begin{proof}[Proof of Theorem \ref{main}. Part 2.]
Suppose $g_S\ge 2$.
Let $\cU_n^r$ be the subset given in Proposition \ref{rob-ks},
$\cV_n$ be the subset given in Proposition \ref{pro:bir},
$\cR=\bigcap_n \cV_n$, and $\cR_{\ast}= \cR \cap \cU_{\ast}$ be the same residual subset of $\cU_{\ast}$ as in the first part of the proof. Applying Proposition \ref{pro.LCS} to the periodic center leaves of a map $\Phi\in \cR_{\ast}$, one has the dichotomy that either every hyperbolic periodic point in that leaf admits (leafwise) transverse homoclinic points, or the stable and unstable manifolds of every hyperbolic periodic points are dense on the periodic center leaf. To prove the theorem, we need the following:\\

\noindent{\bf Claim.} 
Let $n\ge 1$, $\cW_{n}$ be the set of maps $\Phi\in \cV_n \cap \cU_{\ast}$ such that every hyperbolic periodic point $p\in P_n(\Phi)$ admits transverse homoclinic points. 
Then $\cW_{n}$  is a $C^1$-open and $C^r$-dense subset of $\cU_{\ast}$.

\begin{proof}[Proof of Claim]
Since $\cV_n \subset \cU_n^r$, it follows directly from Proposition \ref{rob-ks} the set $P_n(\Phi)$ is finite and varies continuously on $\cV_n \cap \cU_{\ast}$. Moreover, the existence of  transverse homoclinic points for all hyperbolic periodic points in $P_n(\Phi)$ is a $C^1$-open condition in $\cV_n \cap \cU_{\ast}$.
Therefore, the subset $\cW_{n}$ is $C^1$-open by its definition.
Next we will prove that $\cW_{n}$ is also $C^r$-dense in $\cU_{\ast}$.

Let $\Phi\in \cU_{\ast}$ be fixed. Pick a $C^r$-small neighborhood $\cX_n \subset \cV_n$ of $\Phi$ on which the function $\Psi\in \cX_n \mapsto |P_n(\Psi)|$  is constant. Name the hyperbolic ones in $P_n(\Psi)$ by $p_i(\Psi)$, $1\le i \le I_n$ for some $I_n \ge 0$. Recall that $B_{C^r}(\Phi, \delta)$ is the $\delta$-ball of maps $\Psi\in \Diff^r_{\Omega}(M)$ with $C^r$-distance $d_{C^r}(\Psi, \Phi)<\delta$. 
Pick $\eps>0$ such that $B_{C^r}(\Phi, (1+I_n)\eps ) \subset \cX_n$ and a map $\Psi_0\in B_{C^r}(\Phi, \eps) \cap \cR$. 
Then the restriction $\Psi_0^n: \cF^c_{\Psi_0}(p_i) \to \cF^c_{\Psi_0}(p_i)$
satisfies $\Psi_0^n|_{\cF^c_{\Psi_0}(p_i)} \in \cG^4_{\Omega|_{\cF^c_{\Psi_0}(p_i)}}(\cF^c_{\Psi_0}(p_i))$ for each $1\le i \le I_n$.

Case 1. $|P_h(\Psi_0^n|_{\cF^c_{\Psi_0}(p_i)})|> 2g_S -2$ for each $1\le i \le I_n$. 
It follows from Proposition \ref{pro.LCS}  that the hyperbolic periodic point $p_i$ has (leafwise) transverse homoclinic points, $1\le i\le I_n$. Therefore, $\Psi_0 \in \cW_n$.

Case 2. $|P_h(\Psi_0^n|_{\cF^c_{\Psi_0}(p_j)})|= 2g_S -2$ for some (or all) $1\le i\le I_n$.
Relabeling these points if necessary, we assume this equality holds for each $1\le j\le J_n$
for some $1\le J_n \le I_n$.
In the following we will construct a map $\Psi_1 \in B_{C^r}(\Psi_0, \eps)$ for which the point $p_1$ admits transverse homoclinic points. Moreover, the perturbation $\Psi_1$ is $C^r$-close enough to $\Psi_0$ such that the existing transverse homoclinic points for hyperbolic periodic points in $P_{n}(\Psi_0)$ persist under the perturbation $\Psi_1$. Since $\cR_{\ast}$  is $C^r$-dense in $\cU_{\ast}$ and the existence of  transverse homoclinic points is $C^1$-open, we can assume $\Psi_1\in \cR_{\ast}$, too. Then by induction, for each $2\le j \le J_n$, we obtain a map  $\Psi_j \in B_{C^r}(\Psi_{j-1}, \eps)\cap \cR_{\ast}$ close enough to $\Psi_{j-1}$  for which the point $p_j$ admits transverse homoclinic points and the existing transverse homoclinic points for hyperbolic periodic points in $P_{n}(\Psi_{j-1})$ persist under the perturbation $\Psi_j$.
It follows that every hyperbolic periodic point in $P_n(\Psi_{J_n})$ admits transverse homoclinic points. That is,
$\Psi_{J_n} \in \cW_n \cap B_{C^r}(\Phi, (1+I_n)\eps) \subset \cW_n \cap \cX_n$. 
Since  $\cX_n$ is an arbitrarily chosen neighborhood of an arbitrarily chosen map $\Phi \in \cU_{\ast}$, it follows that $\cW_n$ is $C^r$-dense. This will complete the proof of the claim.

The construction of  the perturbation $\Psi_1$ follows the approach in \cite[\S4]{XZ06},
combining Proposition \ref{pro.LCS} on generic surface diffeomorphisms.

For $\eps>0$ given as above, pick $\delta>0$ (much smaller than $\eps$) such that 
for any points $x, y \in M$ with $d(x,y)<\delta$, 
$\cF_{\Psi_0}^s(x, 3\eps)$ and $\cF^{cu}_{\Psi_0}(y, 3\eps)$ intersect at a unique point,
and $\cF_{\Psi_0}^{cs}(x, 3\eps)$ and $\cF^{u}_{\Psi_0}(y, 3\eps)$  intersect at a unique point.
Applying Proposition \ref{pro.denseCL}, 
we can pick another periodic center leaf, say $\cF^c_{\Psi_0}(\hat{p})$ for some point $\hat p \in B(p_1, \delta)$, 
such that $d(\cF^c_{\Psi_0}(\hat{p}), \cF^c_{\Psi_0}(p_1))<\delta$.
Let $\hat{n}$ be the period of the center leaf $\cF^c_{\Psi_0}(\hat{p})$.
By the choice of $\Psi_0$, we have $\Psi_0^{\hat{n}}|_{\cF^c_{\Psi_0}(\hat{p})} \in \cG^4_{\Omega|_{\cF^c_{\Psi_0}(\hat{p})}}(\cF^c_{\Psi_0}(\hat{p}))$.

The initial choice of the point $\hat{p}$ in the periodic center leaf  $\cF^c_{\Psi_0}(\hat{p})$ might be nonperiodic.
Since $g_S \ge 2$, it follows from Proposition \ref{pro.LCS} that $|P_h(\Psi_0^{\hat{n}}|_{\cF^c_{\Psi_0}(\hat{p})})|\ge 2g_S -2 \ge 2$. In particular, there do exist hyperbolic periodic points 
on $\cF^c_{\Psi_0}(\hat{p})$. Let $q$ be such a hyperbolic periodic point on $\cF^c_{\Psi_0}(\hat{p})$ and $m$ be the minimal period of $q$. Note that $m$ can be much larger than the number $n$. 
In the following we will use the point $q$ instead of $\hat{p}$ as the marked point on the center leaf  $\cF^c_{\Psi_0}(\hat{p}) = \cF^c_{\Psi_0}(q)$.

Pick a point $\hat{q}\in \cF^c_{\Psi_0}(p_1)$ with $d(q, \hat{q})<\delta$.
Then $\cF_{\Psi_0}^s(q, \eps)$  and $\cF^{cu}_{\Psi_0}(\hat{q}, \eps)$ intersect at a unique point, say $v$. That is,
$v\in \cF_{\Psi_0}^s(q, \eps) \cap \cF^u_{\Psi_0}(x, \eps)$ for some $x\in \cF^c_{\Psi_0}(\hat{q}, \eps) \subset \cF^c_{\Psi_0}(p_1)$.
Similarly, $\cF_{\Psi_0}^u(q, \eps)$ and $\cF^{cs}_{\Psi_0}(\hat{q}, \eps)$ intersect at a unique point, say $w$. That is,
$w\in \cF_{\Psi_0}^u(q, \eps) \cap \cF^s_{\Psi_0}(y, \eps)$ for some $y\in  \cF^c_{\Psi_0}(\hat{q}, \eps)\subset \cF^c_{\Psi_0}(p_1)$.
Since $|P_h(\Psi^n|_{\cF^c_{\Psi_0}(p_1)})|= 2g_S -2$,
it follows from Proposition \ref{pro.LCS}  
that the leafwise stable and unstable manifolds $W^{s, u}(p_1, \Psi_0^n|_{\cF^c_{\Psi_0}(p_1)})$ of $p_1$ are dense on the whole leaf $\cF^c_{\Psi_0}(p_1)$.
Therefore, we can pick
\begin{enumerate}
\item a sequence of points $x_j \in W^u(p_1, \Psi_0^n|_{\cF^c_{\Psi_0}(p_1)})$ that converge to $x$

\item a sequence of points $y_j \in W^s(p_1, \Psi_0^n|_{\cF^c_{\Psi_0}(p_1)})$ that converge to $y$. 
\end{enumerate}

Note that  $\Psi_0^{-kn}(v)$ and $\Psi_0^{kn}(w)$ converge to the center leaf $\cF^c_{\Psi_0}(p_1)$ as $k\to +\infty$ 
and $\Psi_0^{km}(v)$ and $\Psi_0^{-km}(w)$ converge to the center leaf $\cF^c_{\Psi_0}(q)$ as $k\to +\infty$. 
These two points being non-recurrence  makes the $C^r$-perturbations around these two points straightforward. This is similar to setting as in \cite[\S4]{XZ06}. Following the same argument, one can find a $C^r$-small perturbation $\Psi_1$ of $\Psi_0$ supported
on two disjoint small neighborhoods of $v$ and $w$, respectively, such that
\begin{enumerate} 
\item $v\in \cF_{\Psi_1}^s(q, \eps) \cap \cF^u_{\Psi_1}(x_j, \eps)$ for some $x_j$ sufficiently close to $x$,

\item $w\in \cF_{\Psi_1}^u(q, \eps) \cap \cF^s_{\Psi_1}(y_j, \eps)$ for some $y_j$ sufficiently close to $y$.
\end{enumerate}
Note that $\Psi_1= \Psi_0$ on both center leaves  $\cF^c_{\Psi_0}(p_1)$ and $\cF^c_{\Psi_0}(q)$.
It follows that  $v\in W_{\Psi_1}^s(q) \cap W^u_{\Psi_1}(p_1)$ and $w\in W_{\Psi_1}^u(q) \cap  W_{\Psi_1}^{s}(p_1)$.
That is, there is a heteroclinic cycle between the two hyperbolic periodic points $p_1$ and $q$ for the perturbed map $\Psi_1$.
Making a further perturbation if necessary, we may assume that the heteroclinic intersections at both $v$ and $w$
are transverse. Then it follows from the Lambda Lemma that there are transverse homoclinic points 
for the hyperbolic periodic point $p_1$. This completes the construction of the perturbation $\Psi_1$ 
and the proof of the claim.
\end{proof}

It follows from Claim that $\cW_{n}$ is $C^1$-open and $C^r$-dense for each  $n\ge 1$.
Then the set $\cR_{\ast}':=\cR_{\ast} \cap (\cap_{n\ge 1}\cW_n)$ is a residual subset of $\cU_{\ast}$. Moreove,  for each $\Phi \in \cR_{\ast}'$, every hyperbolic periodic point of $\Phi$ admits transverse homoclinic points. This completes the proof of Theorem \ref{main}.
\end{proof}

\section*{Acknowledgment}

The author express his sincere gratitude to the anonymous referee for his/her invaluable
comments, which help him  to improve the presentation of the paper significantly.

\end{document}